\documentclass[12pt,a4paper]{article}
\usepackage{amssymb}
\usepackage{amsmath}
\usepackage{amsthm}
\usepackage{indentfirst}
\usepackage{ot1patch}
\usepackage{psfrag}
\usepackage{rotate}
\usepackage{mathtools}
\usepackage{hyperref}

\newtheorem{tw}{Theorem}[section]
\newtheorem{pr}[tw]{Proposition}
\newtheorem{lm}[tw]{Lemma}
\newtheorem{cor}[tw]{Corollary}

\theoremstyle{definition}

\DeclareMathOperator{\Irr}{Irr}

\author{Magdalena Jankowska\\
	magdalena.jankowska@student.ukw.edu.pl
	
	\smallskip\\
	
	\L ukasz Matysiak\\
	lukmat@ukw.edu.pl
	
	\smallskip\\
	
	Kazimierz Wielki University\\
	Bydgoszcz, Poland \\
	}
\title{A polynomial composites and monoid domains as algebraic structures and their applications}

\begin{document}

\maketitle

\begin{abstract}
This paper contains the results collected so far on polynomial composites in terms of many basic algebraic properties. Since it is a polynomial structure, results for monoid domains come in here and there. The second part of the paper contains the results of the relationship between the theory of polynomial composites, the Galois theory and the theory of nilpotents. The third part of this paper shows us some cryptosystems. We find generalizations of known ciphers taking into account the infinite alphabet and using simple algebraic methods. We also find two cryptosystems in which the structure of Dedekind rings resides, namely certain elements are equivalent to fractional ideals. Finally, we find the use of polynomial composites and monoid domains in cryptology.
\end{abstract}

\begin{table}[b]\footnotesize\hrule\vspace{1mm}
	Keywords: cryptology, domain, field, monoid, monoid domain, polynomial composites.\\
2010 Mathematics Subject Classification:
Primary 13B25, Secondary 13B05, 11T71.
\end{table}

\section{Introduction}

Let $\mathbb{N}=\{1, 2, \dots\}$, $\mathbb{N}_0=\{0, 1, 2, \dots\}$.
By a ring we mean a commutative ring with unity. Let $R$ be a ring.
Denote by $R^{\ast}$ the group of all invertible elements of $R$. The set of all irreducible elements in $R$ will be denoted by $\Irr R$.
By a domain we mean a commutative ring with unity without zero divisors.
An element $r\in R$ is called nilpotent if there is $n\in\mathbb{N}$ such that $r^n=0$.  

\medskip

The most important motivation for writing this paper is to quote the most important results related to polynomial composites, their algebraic place in mathematics and their application in cryptology. 

\medskip

D.D.~Anderson, D.F.~Anderson, M. Zafrullah in \cite{1} called object $A+XB[X]$ as a composite, where $A$ be a subdomain of the field $B$. If $B$ be a domain and $M$ be an additive cancellative monoid (a semigroup with neutral element and cancellative property) we can define a monoid domain $B[M]=\{a_0X^{m_0}+\dots +a_nX^{m_n}: a_0, \dots , a_n\in B, m_1, \dots m_n\in M\}$. If $M=\mathbb{N}_0$, then $B[M]=B[X]$.
Monoid domains appear in many works, for example \cite{2}, \cite{9}.

\medskip

There are a lot of works where composites are used as examples to show some properties. But the most important works are presented below.

\medskip

In 1976 \cite{y1} authors considered the structures in the form $D+M$, where $D$ is a domain and $M$ is a maximal ideal of ring $R$, where $D\subset R$. Later (\ref{t2}), we could prove that in composite in the form $D+XK[X]$, where $D$ is a domain, $K$ is a field with $D\subset K$, that $XK[X]$ is a maximal ideal of $K[X]$. 
Next, Costa, Mott and Zafrullah (\cite{y2}, 1978) considered composites in the form $D+XD_S[X]$, where $D$ is a domain and $D_S$ is a localization of $D$ relative to the multiplicative subset $S$. 
In 1988 \cite{y5} Anderson and Ryckaert studied classes groups $D+M$.
Zafrullah in \cite{y3} continued research on structure $D+XD_S[X]$ but 
he showed that if $D$ is a GCD-domain, then the behaviour of $D^{(S)}=\{a_0+\sum a_iX^i\mid a_0\in D, a_i\in D_S\}=D+XD_S[X]$ depends upon the relationship between $S$ and the prime ideals $P$ od $D$ such that $D_P$ is a valuation domain (Theorem 1, \cite{y3}).
Fontana and Kabbaj in 1990 (\cite{y4}) studied the Krull and valuative dimensions of composite $D+XD_S[X]$. 
In 1991 there was an article (\cite{1}) that collected all previous composites and the authors began to create a theory about composites creating results. In this paper, the structures under consideration were officially called as composites. 
After this article, various minor results appeared. But the most important thing is that composites have been used in many theories as examples. In \cite{mm1} we have a general definition of composite as polynomial composite.

\medskip

In the second section we can find many results about polynomial composites and monoid domains. Basic algebraic properties such as irreducible elements, nilpotents and ideals have been examined.
Theorem \ref{t2} is especially worth noting. In this theorem, for $A\subset B$ be fields, we can note that every nonzero prime ideal polynomial composites is maximal, every prime ideal different from some maximal ideal of polynomial composite is principal and every polynomial composites are atomic (every element of polynomial composites be a product of finite irreducibles(atoms)). In Theorem \ref{t3} we have an iformation about irreducibles of monoid domain. 
In the second part of the second section we have results about ACCP and atomic properties. Recall, if there is an ascending chain of principal ideals of ring $R$: $I_1\subset I_2\subset I_3\subset $, then there is $n\in\mathbb{N}$ such that $I_n=I_{n+1}=\dots$. A domain with ACCP property is called ACCP-domain. Every ACCP-domain be atomic. In Theorem \ref{Dedekind} it turns out that the polynomial composite of the form $K+XL[X]$ (where $K\subset L$ be fields) is a Dedekind ring. This is a very important class of rings in algebra.

\medskip

In the third section we can find relationships between a theory of polynomial composites and Galois theory. Galois theory contributed greatly to the development of many fields, not only in mathematics. Particularly noteworthy is the solution of three ancient problems of construction with a compass and a straightedge in the 19th century. The results in this section, under different assumptions, boil down to the relationship between field extensions and Noetherian rings.
Recall a Noetherian ring is called a ring with ACCP-property. Equivalence, a ring such that every ideal be a finite generated. In Theorems \ref{tm3} and \ref{tm4} we combine the Magid's results with the current results to create a complete characterization of field extensions using polynomial composites and idempotents. Recall, an element $e\in R$ is called idempotent if $e^2=e$ holds. For example, in $\mathbb{Z}$ we have two trivial idempotents $0$ and $1$. 
 
\medskip

Sections four and five are reminder from \cite{kk1} a generalized RSA cipher and a Diffie-Hellman protocol key exchange. Such a reminder is purposeful because we want to draw attention to the replacement of the finite alphabet with the infinite one and the replacement of classical prime numbers with prime ideals. Such a swap will be extremely difficult for third person to break.

\medskip

In sections six and seven we have cryptosystems which use the structure Dedekind. The former uses this structure in the key, and the latter uses it in two different alphabets. Of course, these ciphers can be generalized to infinite alphabets and ideals.

\medskip

Section eight shows a cryptosystem based on polynomial composites. Section nine shows a cryptosystem based on monoid domains. Note that in the last cryptosystem, in order to break it, the discrete logarithm calculation should be used. At the moment, there is no mathematical way to facilitate the computation of discrete logarithms. We can count using computers, but here the algorithm would consist in checking each successive number, not on a specific indication of the number. And this is a great difficulty in breaking the last cryptosystem.

\section{Polynomial composites and monoid domains}
\label{S1}

In this section we introduce the most important facts about polynomial composites and monoid domains in math.

\medskip

Let's start from the following Lemma which is very easy to proof.

\medskip

\begin{lm}
	\label{l1}
	$T_n$, $T$ be rings, $T_n\subset T$.
\end{lm}

Now let's look at invertible and nilpotent elements.

\begin{pr}
	\label{p2}
	Let $f=a_0+a_1X+\dots + a_nX^n\in T$ for any $n\in\mathbb{N}_0$. Then $f\in T^{\ast}$ if and only if $a_0\in A^{\ast}$ and $a_1, a_2, \dots, a_n$ are nilpotents.
\end{pr}

\begin{proof}
	We know that if $R$ is a ring then $f=a_0+a_1X+\dots + a_nX^n\in R[X]^{\ast}$ if and only if $a_0\in R^{\ast}$ and $a_1, a_2, \dots , a_n$ are nilpotents. In our Proposition we have $a_1, a_2, \dots, a_n$ are nilpotents. Of course we get $a_0\in A^{\ast}$.
\end{proof}

\begin{pr}
	\label{p3}
	Let $f=a_0+a_1X+\dots a_{n-1}X^{n-1}+a_nX^n+\dots +a_mX^m\in T_n$, where $0\leq n\leq m$ and $a_i\in A_i$ for $i=0, 1, \dots, n$ and $a_j\in B$ for $j=n, n+1, \dots, m$.
	\begin{itemize}
		\item[(i) ] $f\in T_n^{\ast}$ if and only if $a_0\in A_0^{\ast}$ and $a_1, a_2, \dots, a_m$ are nilpotents.
		\item[(ii) ] $f$ be a nilpotent if and only if $a_0, a_1, \dots, a_m$ are nilpotents.
	\end{itemize}
\end{pr}

\begin{proof}
	Analogous proof like in Proposition \ref{p2}.
\end{proof}

\begin{pr}
	\label{prpr1}
	Let $B$ be a domain and $f=a_{m_1}X^{m_1}+a_{m_2}X^{m_2}+\dots +a_{m_n}X^{m_n}\in B[M]$, where $m_1, m_2, \dots, m_n\in M$ and $a_{m_1}, a_{m_2}, \dots, a_{m_n}\in B$. 
	\begin{itemize}
		\item[(i) ] $f\in B[M]^{\ast}$ if and only if there exist $m_i\in M$ such that $a_{m_i}\in B^{\ast}$ and $m_i=0$ and for every $m_k\neq m_i$ we have $a_{m_k}$ be nilpotents.
		\item[(ii) ] $f$ be a nilpotent if and only if $a_{m_1}, a_{m_2}, \dots, a_{m_n}$ are nilpotents.
	\end{itemize}
\end{pr}

\begin{proof}
	$(i)$ Assume $f\in B[M]^{\ast}$. Then there exists $g=b_{m'_1}X^{m'_1}+b_{m'_2}X^{m'_2}+\dots +b_{m'_n}X^{m'_n}$, where $m'_1, m'_2, \dots , m'_n\in M$ and $b_{m'_1}, b_{m'_2}, \dots , b_{m'_n}\in B$ such that $fg=1$. Hence there exist $m_i, m_j\in M$ such that $a_{m_i}b_{m_j}X^{m_i+m_j}=1$. We have $a_i\in B^{\ast}$ and $m_i, m_j=0$. The rest of coefficients are nilpotents. On the other side of the proof it is easy.
	
	\medskip
	
	\noindent
	$(ii)$ Obvious.
\end{proof}

Let's recall Theorem from \cite{1} (Theorem 2.9) in a different form.

\begin{tw}
	\label{t1}
	Let $A$ be a subfield of $B$. Consider $D=A+XB[X]$. Then 
	$\Irr D=\{aX, a\in B\}\cup
	\{a(1+Xf(X)),a\in A, f\in B[X], 1+Xf(X)\in\Irr B[X]\}.$
\end{tw}

\begin{tw}
	\label{t2}
	Consider $T=A+XB[X]$, where $A$ be a subfield of $B$; $T_n=A_0+A_1X+A_2X^2\dots +A_{n-1}X^{n-1}+X^nB[X]$, where $A_0\subset A_1\subset A_2\subset\dots \subset A_{n-1}\subset B$ be fields. Then
	\begin{itemize}
		\item[(i) ] every nonzero prime ideal of $T$ ($T_n$, respectively) is maximal;
		\item[(ii) ] every prime ideal $P$ different from $A_1X+A_2X^2+\dots + A_{n-1}X^{n-1}+X^nB[X]$ (in $T_n$) is principal;
		\item[(iii) ] every prime ideal $P$ different from $XB[X]$ (in $T$) is principal;
		\item[(iv) ] $T_n$ is atomic, i.e., every nonzero nonunit of $T$ is a finite product of irreducible elements (atoms);
		\item[(v) ] $T$ is atomic. 
	\end{itemize} 
\end{tw}

\begin{proof}
	\textit{(i)}. We proof for $T_n$. The proof for $T$ will be a corollary.
	
	\medskip
	
	\noindent
	First note that $A_1X+A_2X^2+\dots + A_{n-1}X^{n-1}+X^nB[X]$ is maximal since $T_n/A_1X+A_2X^2+\dots + A_{n-1}X^{n-1}+X^nB[X]\cong A_0$. Let $P$ be a nonzero prime ideal of $T_n$. Now $X\in P$ implies $(T_n/A_1X+A_2X^2+\dots + A_{n-1}X^{n-1}+X^nB[X])^2\subseteq P$ and hence $A_1X+A_2X^2+\dots + A_{n-1}X^{n-1}+X^nB[X]\subseteq P$ so $P=A_1X+A_2X^2+\dots + A_{n-1}X^{n-1}+X^nB[X]$. So suppose that $X\notin P$. Then for $N=\{1, X, X^2, \dots \}$, $P_N$ is a prime ideal in the PID $B[X,X^{-1}]=T_{n,N}$. (In fact, $B[X,X^{-1}]\subseteq R_P$ and $R_P$ is a DVR (discrete valuation ring)). So $P$ is minimal and is also maximal unless $P\subsetneq A_1X+A_2X^2+\dots + A_{n-1}X^{n-1}+X^nB[X]$. But let $k_nX^n+\dots + k_sX^s\in P$ with $k_n\in\mathbb{N}_0$, where $k_n, \dots, k_s\in B$ for any $n, s$. Then $X^{n+1}+k_n ^{-1}k_{n+1}X^{n+2}+\dots +k_n ^{-1}k_sX^s\in P$, so $X\notin P$ implies that $1+k_n ^{-1}k_{n+1}X+\dots +k_n ^{-1}k_sX^{s-n}\in P$, a contradiction. So every nonzero prime ideal is maximal.
	
	\medskip
		
	\noindent
	\textit{(ii)}, (similarly \textit{(iii)}). If $P$ is different from $A_1X+A_2X^2+\dots + A_{n-1}X^{n-1}+X^nB[X]$, then it contains an element of the form $1+a_1X+a_2X^2+\dots +a_{n-1}X^{n-1}+X^nf(X)$, where $a_i\in A_i$ for $i=1, 2, \dots, n-1$ and $f(X)\in B[X]$. Now if $1+a_1X+a_2X^2+\dots +a_{n-1}X^{n-1}+X^nf(X)$ can be factored in $A_1X+A_2X^2+\dots + A_{n-1}X^{n-1}+X^nB[X]$ it can be written as $(1+b_1X+b_2X^2+\dots +b_{n-1}X^{n-1}+X^ng(X))(1+c_1X+c_2X^2+\dots +c_{n-1}X^{n-1}+X^nh(X))$, where $b_i, c_i\in A_i$ for $i=1, 2, \dots, n-1$ and $g(X), h(X)\in B[X]$. Hence $1+a_1X+a_2X^2+\dots +a_{n-1}X^{n-1}+X^nf(X)$ is irreducible in $T_n$ if and only if it is irreducible in $A_1X+A_2X^2+\dots + A_{n-1}X^{n-1}+X^nB[X]$. 
	
	\medskip
	
	\noindent
	Now let $1+a_1X+a_2X^2+\dots +a_{n-1}X^{n-1}+X^nf(X)$ be irreducible in $T_n$ and suppose that $1+a_1X+a_2X^2+\dots +a_{n-1}X^{n-1}+X^nf(X)\mid k(X)l(X)$ in $T_n$. Then $1+a_1X+a_2X^2+\dots +a_{n-1}X^{n-1}+X^nf(X)\mid k(X)l(X)$ in $A_1X+A_2X^2+\dots + A_{n-1}X^{n-1}+X^nB[X]$, and so in $A_1X+A_2X^2+\dots + A_{n-1}X^{n-1}+X^nB[X]$ we have, say $1+a_1X+a_2X^2+\dots +a_{n-1}X^{n-1}+X^nf(X)\mid k(X)$. Then, in $A_1X+A_2X^2+\dots + A_{n-1}X^{n-1}+X^nB[X]$, $k(X)=(1+a_1X+a_2X^2+\dots +a_{n-1}X^{n-1}+X^nf(X))d(X)$. Now $d(X)$ can be written as $d(X)=aX^r(1+a_1X+a_2X^2+\dots +a_{n-1}X^{n-1}+X^np(X))$. If $r>0, d(X)\in T_n$, while if $r=0, k(X)=(1+a_1X+a_2X^2+\dots +a_{n-1}X^{n-1}+X^nf(X))(b(1+b_1X+b_2X^2+\dots +b_{n-1}X^{n-1}+X^np(X))$ and $b\in A_0$ because $k(X)\in T_n$. In either case, $d(X)\in T_n$ and so $1+a_1X+a_2X^2+\dots +a_{n-1}X^{n-1}+X^nf(X)\mid k(X)$ in $T_n$. Consequently, in $T_n$ every irreducible element of the type $1+a_1X+a_2X^2+\dots +a_{n-1}X^{n-1}+X^nf(X)$ is prime.
	
	\medskip
	
	\noindent
	Now since every element of the form $1+a_1X+a_2X^2+\dots +a_{n-1}X^{n-1}+X^nf(X)$ is a product of irreducible elements of the same form and hence is a product of prime elements, it follows that every prime ideal of $P$ different from $A_1X+A_2X^2+\dots + A_{n-1}X^{n-1}+X^nB[X]$ contains a principal prime and hence is actually principal.
	
	\medskip
	
	\noindent
	\textit{(iv)} (similarly \textit{v}). From \textit{(ii)} a general element of $T_n$ can be written as $aX^r(1+a_1X+a_2X^2+\dots +a_{n-1}X^{n-1}+X^nf(X))$, where $a\in B$ (with $a\in A_0$ if $r=0$) and $1+a_1X+a_2X^2+\dots +a_{n-1}X^{n-1}+X^nf(X)$ is a product of primes.
\end{proof}

Now, We give some basic information related to ideals.

\begin{cor}
	\label{p4}
	\begin{itemize}
		\item[(i) ] If $A$ be a field, then $XB[X]$ be an maximal ideal in $T$.
		\item[(ii) ] If $A$ be an integral domain, then $XB[X]$ be an prime ideal in $T$.
		\item[(iii) ] $T/(X)\cong A$.
		\item[(iv) ] $T/B\cong \{0\}$.
		\item[(v) ] Let $A\subset B$ be fields in $T$. $T/(aX)$ be a field for any $a\in B$.
		\item[(vi) ] Let $A\subset B$ be fields in $T$. $T/(a(1+Xf(X)))$ be a field for any $a\in A, f\in B[X]$ such that $1+Xf(X)\in\Irr B[X]$.
	\end{itemize}
\end{cor}

\begin{proof}
	(i) Let $A$ be a field. The proof follows from $T/XB[X]\cong A$. We have $XB[X]$ is a maximal ideal in $T$.
	
	\medskip
	
	(ii) -- (iv) Obvious.
	
	\medskip
	
	(v), (vi) From Theorem 2.9 in \cite{1} $aX$ for any $a\in B$ is an irreducible element. We get $T/(aX)$ be a field. We also have $a(1+Xf(X))$ for any $a\in A, f\in B[X]$ such that $1+Xf(X)\in\Irr B[X]$ is a irreducible element. We have $T/(a(1+Xf(X)))$ be a field.
\end{proof}

\begin{cor}
	\label{p5}
	\begin{itemize}
		\item[(i) ] If $A_0+A_1X+\dots +A_{n-1}X^{n-1}$ be a field (where $A_0\subset A_1\subset A_2\subset\dots \subset A_{n-1}\subset B$), then $X^nB[X]$ be an maximal ideal in $T_n$.
		\item[(ii) ] If $A_0+A_1X+\dots +A_{n-1}X^{n-1}$ be a domain, then $X^nB[X]$ be an prime ideal in $T_n$.
		\item[(iii) ] $T_n/(X)\cong A_0$.
		\item[(iv) ] $T_n/B\cong \{0\}$.
		\item[(v) ] Let $A_0\subset A_1\subset\dots \subset B$ be fields in $T_n$. $T_n/(aX)$ be a field for any $a\in B$.
		\item[(vi) ] Let $A_0\subset A_1\subset\dots \subset B$ be fields in $T_n$. $T_n/(a(1+a_1X+a_2X^2+\dots +a_{n-1}X^{n-1}+ X^nf(X)))$ be a field for any $a\in B, a_i\in A_i (i=1, 2, \dots , n-1), f\in B[X]$ such that $1+Xf(X)\in\Irr B[X]$.
	\end{itemize}
\end{cor}

\begin{proof}
	The proof is similarly to proof of Corollary \ref{p4}.
\end{proof}

\begin{tw}
	\label{t3}
	Consider $T=A+XB[X]$, where $A$ be a subfield of $B$; $T_n=A_0+A_1X+A_2X^2\dots +A_{n-1}X^{n-1}+X^nB[X]$, where $A_0\subset A_1\subset A_2\subset\dots \subset A_{n-1}\subset B$ be fields. Then
	\begin{itemize}
		\item[(i) ] $f\in\Irr T$ if and only if $f\in\Irr B[X], f(0)\in A$.
		\item[(ii) ] $f\in\Irr T_n$ if and only if $f\in\Irr B[X], a_i\in A_i$, where $f=a_0+a_1X+\dots a_{n-1}X^{n-1}+a_nX^n+\dots +a_{m}X^m$ with $a_i\in A_i$ for $i=0, 1, \dots n-1$ and $a_n, a_{n+1}, \dots , a_m\in B (n<m)$.
	\end{itemize}
\end{tw}

\begin{proof}
	\textit{(i)}. 
	Suppose that $f\notin\Irr B[X]$ or $f(0)\notin A$.
	If $f(0)\notin A$, then $f\notin T$, so $f\notin\Irr B[X]$.
	Now, assume that $f\notin\Irr B[X]$.
	Then $f=gh$, where $g, h\in B[x]\setminus B$.
	Let $g=a_0+a_1X+\dots +a^nX^n, h=b_0+b_1X+\dots + b_mX^m$.
	We have $f=(a_0+a_1X+\dots +a^nX^n)(b_0+b_1X+\dots + b_mX^m)$.
	Then
	$f=\big(1+\dfrac{a_1}{a_0}X+\dots + \dfrac{a_n}{a_0}X^n\big)
	(a_0b_0+a_0b_1X+\dots +a_0b_mX^m)$,
	where $a_0b_0=f(0)\in A$.
	Now, suppose that $f\notin\Irr T$.
	If $f\notin T$, then $f(0)\notin A$.
	Now, assume that $f\in T$.
	Then we have $f=gh$, where $g, h\in T\setminus A$.
	This implies $g, h\in B[x]\setminus B$.
	
	\medskip
	
	\noindent
	\textit{(ii)} Occur in the same way as in \textit{(i)}.
	
	\medskip
	
\end{proof}

In \cite{2}, Lemma 6.4 we have informations about irreducible element in monoid domain $D[S]$, where $D$ be a domain, and $S$ be a submonoid of $\mathbb{Q}_+$ . I present a generalized Proposition.

\begin{pr}
	\label{x1}
	Let $B$ be an integral domain with quotient field $K$ and $M$ a monoid with quotient group $G\neq M$. Assume that $B$ contains prime elements $p_1, p_2, \dots , p_{r-1}$. Assume that $M$ is integrally closed and each nonzero element of $G$ is type $(0, 0, \dots )$ ($G$ satisfies the ascending chain condition on cyclic subgroups).
	Consider $m_1, m_2, \dots , m_r\in M$ such that $m_1\in\Irr M$ and $m_2, m_3, \dots , m_r\notin m_1 + M$. Then $p_{r-1}X^{m_{r}}-\dots -p_2X^{m_3}-p_1X^{m_2} - X^{m_1}\in\Irr B[M]$.
\end{pr}

\begin{proof}
	Let $\leq$ be a total order on $G$. We may assume that $m_r<m_{r-1}<\dots m_2<m_1$. Suppose that $p_{r-1}X^{m_{r}}-\dots -p_2X^{m_3}-p_1X^{m_2} - X^{m_1}=fg$ with $f,g\in B[M]$. Write $f=a_1X^{t_1}+\dots a_mX^{t_m}$ and $g=b_1X^{k_1}+\dots +b_nX^{k_n}$ in canonical form, where $t_1< \dots < t_m$ and $k_1< \dots < k_n$. First assume that either $f$ or $g$ is a monomial, say $f=aX^t$.Then $a\in B^{\ast}, m_1=t+k_n, m_2=t+k_1, m_3=t+k_2, \dots , m_r=t+k_{r-1}$. Since $m_1\in\Irr M$, either $t$ or $k_n$ is invertible in $M$. If $k_n$ is invertible, then $m_2=t+k_1=(m_1-k_n)+k_1\in m_1+M, m_3=t+k_2=(m_1-k_n)+k_2\in m_1+M, \dots , m_r\in m_1+M$, a contradiction. Thus $t$ is invertible in $M$, and hence $f$ is a unit in $B[M]$. Thus we may assume that $f$ and $g$ are not monomials. Now consider the reduction of $p_{r-1}X^{m_{r}}-\dots -p_2X^{m_3}-p_1X^{m_2} - X^{m_1}=fg$ modulo the ideal $(p_1, p_2, \dots , p_{r-1})$. Then $(-1+(p_1, p_2, \dots , p_{r-1})=((a_m+(p_1, p_2, \dots , p_{r-1}))X^{t_m})((b_n+(p_1, p_2, \dots , p_{r-1}))X^{k_n})$.  This means that $a_1+(p_1, p_2, \dots , p_{r-1})=b_1+(p_1, p_2, \dots , p_{r-1})=(p_1, p_2, \dots , p_{r-1})$. In this case $c_1p_1+\dots c_{r-1}p{r-1}=a_1b_1\in (p_1, \dots , p_{r-1})^2$, a contradiction. Thus $p_{r-1}X^{m_{r}}-\dots -p_2X^{m_3}-p_1X^{m_2} - X^{m_1}\in\Irr B[M]$. 
\end{proof}

\begin{pr}
	$B[M]/(p_{r-1}X^{m_r}-\dots - p_1X^{m_2}-X^{m_1})$ be a field, where $B$ be a domain, $p_1, p_2, \dots, p_{m_r}\in B, m_1, m_2, \dots, m_r\in M$ with $m_1\in\Irr M$, $m_2, m_3, \dots, m_r\notin m_1+M$.
\end{pr}

\begin{proof}
	It follows from Proposition \ref{x1}.
\end{proof}

Recall that Noetherian rings satisfy the ACCP condition. Almost every mathematician has encountered such rings. For example, $\mathbb{Z}$ is a Noetherian ring. Below are the results of ACCP properties in polynomial composites and monoid domains.

\begin{pr}
	\label{p2.1}
	Let $A$ be an integral domain, $B$ be a field such that $A\subset B$. Let $R$ be a ring with $A[X]\subseteq R\subseteq B[X]$. Then $R$ has ACCP if and only if $R\cap B$ has ACCP and for each ascending chain of polynomials $f_1R\subseteq f_2R\subseteq f_3R\subseteq\dots $ where $f_i\in R$ all have the same degree, then there is $d\in (R\cap B)\setminus\{0\}$ such that $df_i\in (R\cap B)[X]$. 
\end{pr}

\begin{proof}
	\cite{mm2}, Proposition 2.1.
\end{proof}

Proposition \ref{p2.2} shows that between $A[X]$ and $A+XB[X]$ we can find a structure which satysfying ACCP condition. 

\begin{pr}
	\label{p2.2}
	Let $A$ be an integral domain, $B$ be a field such that $A\subset B$. Let $C$ be a domain such that $A[X]\subseteq C\subseteq A+XB[X]$. Suppose that for each $n\in\mathbb{N}_0$, there exists $a_n\in A\setminus\{0\}$ for all $f\in C$ with $\deg f\leq n$. Then $C$ has ACCP if and only if $A$ has ACCP.
\end{pr}

\begin{proof}
\cite{mm2}, Proposition 2.2.
\end{proof}

The above Proposition is not obvious for arbitrary composition. This would be a valuable remark, as it would allow we to choose the smallest possible composite.

\medskip

\noindent
{\bf Question:} For subdomains $A_0, A_1, \dots, A_{n-1}$ of a field $B$, is the Proposition \ref{p2.2} valid for such domain $C$ satysfying $A_0[X]\subseteq C\subseteq A_0+A_1X+\dots + A_{n-1}X^{n-1} + X^nB[X]$, where the condition $A_0\subset A_1\subset\dots \subset A_{n-1}\subset B$ holds or not?

\medskip

Kim in \cite{2} proved very serious fact about ACCP monoid domain.

\begin{lm}
	\label{l2.3}
	Let $A$ be a domain. Then $A$ has ACCP if and only if $A[X]$ has ACCP.
\end{lm}

\begin{proof}
	\cite{2}, Corollary 2.2. Can be easily proved by comparing degrees.
\end{proof}

It also turn out that ACCP property moves between $A$ and $A+XB[X]$. This is important because we do not have to choose a general polynomial, and we can limit the inclusion to the smallest composite needed. Such a significant limitation of a polynomial to a composite is important, e.g. in Galois theory in commutative rings.

\begin{tw}
	\label{p2.4}
	Let $A$ be an integral domain, $B$ be a field such that $A\subset B$. An $A$ has ACCP if and only if $A+XB[X]$ has ACCP.
\end{tw}

\begin{proof}
	From Proposition \ref{p2.2} we have $A[X]\subseteq A+XB[X]\subseteq A+XK[X]$, where $K$ be a qoutient field of $B$. We can to prove that for each $n\geq 0$, there exists $a_n\in A\setminus\{0\}$ for all $f\in A+XB[X]$ with $\deg f\leq n$. Because $A$ has ACCP then from Proposition \ref{p2.2} we get $A+XB[X]$ has ACCP. Conversly, because $A+XB[X]$ has ACCP then $A$ has ACCP.	
\end{proof}

The next facts are the conclusions of Theorem \ref{p2.4}.

\begin{cor}
	\label{c2.5}
	Let $A_0, A_1, \dots, A_{n-1}$ be subdomains of a field $B$ such that $A_0\subset A_1\subset \dots \subset B$. Let $C$ be a domain with $A_0[X]\subseteq C\subseteq A_0+A_1X+\dots +A_{n-1}X^{n-1}+X^nB[X]$. Suppose that for each $n\geq 0$, there exists $a_n\in A_0\setminus\{0\}$ for all $f\in C$ with $\deg f\leq n$. Then $C$ has ACCP if and only if $A_0$ has ACCP.
\end{cor}

\begin{cor}
	\label{c2.6}
	Let $A_0, A_1, \dots, A_{n-1}$ be subdomains of a field $B$ such that $A_0\subset A_1\subset \dots \subset B$. An $A_0$ has ACCP if and only if $A_0+A_1X+\dots +A_{n-1}X^{n-1}+X^nB[X]$ has ACCP.
\end{cor}

\medskip

Next Lemmas coming from Kim \cite{2} are results about ACCP properties in monoid domains.

\begin{lm}
	\label{l2.9}
	Let $S\subseteq T$ be an extension of torsion-free cancellative monoids. If $T$ satisfies ACCP and $T^{\ast}\cap S=S^{\ast}$, then $S$ satisfies the ACCP.
\end{lm}

\begin{proof}
	\cite{2} Proposition 1.2. (1).
\end{proof}

\begin{lm}
	\label{l2.10}
	Let $D$ be an integral domain, $S$ a torsion-free cancellative additive monoid, and $D[S]$ the monoid domain. If $D[S]$ satisfies ACCP, then $D$ and $S$ satisfy ACCP.
\end{lm}

\begin{proof}
	\cite{2}, Proposition 1.5.
\end{proof}

Next Theorem is the answer about question from Kim \cite{2} Question 1.6. In \cite{2} Proposition 1.5 (1) we have an implication. Kim asked that are the sufficient conditions in \cite{2} Proposition 1.5 (1) for the monoid domain to satisfy ACCP, necessary.

\begin{tw}
	\label{p2.11}
	Let $A$ be an integral domain and $B$ be a field such that $A\subset B$ and $A[S]^{\ast}=B[S]^{\ast}$. Let $S$ be a torsion-free cancellative monoid. Both $A$ and $B[S]$ satisfy ACCP if and only if $A[S]$ satisfies ACCP.
\end{tw}

\begin{proof}
	($\Rightarrow$) The proof is similar to \cite{2}, Proposition 1.5.
	
	\medskip
	
	($\Leftarrow$) From Lemma \ref{l2.10}, since $A[S]$ has ACCP, then $A$ has ACCP.
	Now, consider $f_1, f_2, \dots \in B[S]$ such that $\dots, f_3\mid f_2, f_2\mid f_1$. Without loss of generality, we can assume that $f_1, f_2, \dots \in\Irr B[S]$ because every ACCP-domain is atomic. Since $A^{\ast}=B^{\ast}$, so $f_1, f_2, \dots \in\Irr A[S]$. By assumption $A[S]$ has ACCP, so there exists $n\geqslant 1$ such that $f_n\mid f_{n-1}, \dots, f_3\mid f_2, f_2\mid f_1$. We get $(f_1)\subseteq (f_2)\subseteq \dots \subseteq (f_n)=(f_{n+1})=\dots $ in $B[S]$ which is stationary.\\
\end{proof}

Recall that each ACCP-domain is atomic. Hence, all previous results about the ACCP-domains hold for the atomic domains. We complete the knowledge about the atomicity condition in monoid domains.

\begin{lm}
	\label{l3.1}
	Let $D$ be an integral domain, $S$ a torsion-free cancellative monoid, and $D[S]$ the monoid domain. If $D[S]$ be atomic, then $D$ and $S$ be atomic.
\end{lm}

\begin{proof}
	\cite{2}, Proposition 1.4.
\end{proof}

Next Theorem is similarly to \ref{p2.11}.

\begin{tw}
	\label{p3.2}
	Let $A$ be an integral domain and $B$ be a field such that $A\subset B$ with $A[S]^{\ast}=B[S]^{\ast}$. Let $S$ be a torsion-free cancellative monoid. 
	Both $A$ and $B[S]$ be atomic if and only if $A[S]$ be atomic. 
\end{tw}

\begin{proof}
	($\Rightarrow$) Since $B[S]$ be atomic, then consider $f=g_1g_2\dots g_n\in B[S]$, where $g_1, g_2, \dots, g_n\in\Irr B[S]$. Hence from assumption we have $g_1, g_2, \dots, g_n\in\Irr A[S]$. Then $A[S]$ is atomic.
	
	\medskip
	
	($\Leftarrow$) From Lemma \ref{l3.1} since $A[S]$ be atomic, then $A$ and $S$ be atomic. Now consider $f=g_1g_2\dots g_n\in A[S]$, where $g_1, g_2, \dots, g_n\in\Irr A[S]$, because $A[S]$ be atomic. Then $g_1, g_2, \dots, g_n\in\Irr B[S]$, hence $B[S]$ be atomic. 
\end{proof}

Anderson, Anderson and Zafrullah asked in \cite{0} (Question 1) is $R[X]$ atomic when $R$ is atomic. I say no. I have no example but we can deduce from well known facts:

\medskip

\noindent
Suppose that $R[X]$ is not atomic. We want to get $R$ is not atomic. Since $R[X]$ is not atomic then $R[X]$ has no ACCP. Hence $R$ has no ACCP which it does not imply $R$ is not atomic because there exists an example atomic domain which is not ACCP.

\medskip

\noindent
Converse, if $R$ is not atomic, then $R$ has no ACCP. Hence $R[X]$ has no ACCP which it does not imply $R[X]$ is atomic.

\medskip

In \cite{mm3} we have another results about polynomial composites. Various properties have been investigated: BFD (bounded factorization domain), HFD (half factorial domain), idf (each nonzero element of domain has at most a finite number of nonassociate irreducible divisors), FFD (finite factorization domain), S-domain (for each height-one prime ideal $P$ of domain, height of $P[X]$ is equal to $1$ in polynomial ring over domain), Hilbert ring (every prime ideal of domain is an intersection of maximal ideals of that domain).

\medskip

Theorem \ref{Dedekind} says that, under some assumption, a polynomial composite has the structure of Dedekind rings. Dedekind rings are a very important class of rings in algebra. There are a lot of work, the results associated with it. On the basis of the Dedekind structure, I developed cryptosystems with the Dedekind structure in Sections \ref{RR1} and \ref{RR2}.

\begin{tw}
	\label{Dedekind}
	Let $K\subset L$ be a finite fields extension. Then $K+XL[X]$ be a Dedekind domain.
\end{tw}

\begin{proof}
	By \cite{1}, Theorem 2.7 every nonzero prime ideal is a maximal.
	By \cite{mm3} Proposition 3.1 we have $K+XL[X]$ is integrally closed.
	By Theorem 3.2 \cite{mm4}  $K+XL[X]$ is noetherian domain.
	Hence $K+XL[X]$ be a Dedekind domain.
\end{proof}

In the following Proposition, we provide the most important and fundamental facts about the structure of Dedekind.

\begin{pr}
	\label{pr14}
	Let $K\subset L$ be an extension fields and let $T=K+XL[X]$.
	\begin{itemize}
		\item[(a) ] If $P$ be a nonzero prime ideal of $T$ and $P'=\{x\in T_0; xP\subset T\}$, then $PP'=T$.
		\item[(b) ] Every nonzero ideal of $T$ has an unambiguous representation in the form product of prime ideals.
		\item[(c) ] Every nonzero ideal of $T$ is invertible.
		\item[(d) ] If $I$ is an ideal of $T$, then $T/I$ is a principal ideal domain.
		\item[(e) ] $Cl(T)$ (a group of class of invertible ideals) be isomorphic to $Pic(T)$ (a group of class of invertible modules).
		\item[(f) ] If $M$ be a finite generated torsion-free $T$-module, then $M\cong I_1\oplus I_2\oplus\dots \oplus I_k$, where $I_1$, $I_2$, $\dots$, $I_k$ are nonzero ideals of $T$ and $k$ is a rang of $M$. Moreover
		$$M\cong T^{k-1}\oplus I_1I_2\dots I_k.$$
		\item[(g) ]  If $M$ be a finite generated $T$-module, then 
		$$M\cong T^{k-1}\oplus I\oplus \bigoplus_{(P_i, n_i)} T/P_i^{n_i},$$
		where $k=\dim_{T_0}(M\otimes_T T_0)$, $I\subset T$, $I$ is unambiguously, with the accuracy to isomorphism, a designated ideal, $P_i$ are nonzero prime ideals of $T$, $n_i>0$, and a finite set of pair $(P_i, n_i)$ is designated unambiguously.
	\end{itemize}
\end{pr}

\section{Relationships between polynomial composites and certain types of fields extensions}

Let $K\subset L$ be a fields extension. Let's build a polynomial composites $K+XL[X]$. In this section, we will answer the question of whether there are relationships between field extensions and polynomial composites.

\medskip

All my considerations began with the Theorem \ref{01} below. This Proposition motivated me to further consider polynomial composites $K+XL[X]$ in a situation where the extension of fields $K\subset L$ is algebraic, separable, normal and Galois, respectively.

\begin{tw}
	\label{01}
	Let $K\subset L$ be a field extension. Put $T=K+XL[X]$. Then
	$T$ is Noetherian if and only if $[L\colon K]<\infty$.
\end{tw}

\begin{proof}
	($\Rightarrow$) 
	Since $XL[X]$ is a finitely generated ideal of $K+XL[X]$, it follows from \cite{mm4} Lemma 3.1 that $[L\colon K]<\infty$. Thus, $L[X]$ is module-finite over the Noetherian ring $K+XL[X]$.
	
	\medskip
	
	($\Leftarrow$)
	$L[X]$ is Noetherian ring and module-finite over the subring $K+XL[X]$. This is the situation covered by P.M. Eakin's Theorem \cite{zzz}.
\end{proof}

Every Propositions and Theorems of this section we can find in \cite{mm4}.

\begin{pr}
	\label{02}
	Let $K\subset L$ be a fields extension such that $L^{G(L\mid K)}=K$. Put $T=K+XL[X]$. 
	$T$ is Noetherian if and only if $K\subset L$ be an algebraic extension.
\end{pr}

\begin{pr}
	\label{04}
	Let $K\subset L$ be fields extension such that $K$ be a perfect field and assume that any $K$-isomorphism $\varphi\colon M\to M$, where $\varphi(L)=L$ holds for every field $M$ such that $L\subset M$.
	Put $T=K+XL[X]$. 
	$T$ be a Noetherian if and only if $K\subset L$ be a separable extension.
\end{pr}

\begin{pr}
	\label{06}
	Let $K\subset L$ be fields extension. Assume that if a map $\varphi\colon L\to a(K)$ is $K$-embedding, then $\varphi (L)=L$. 
	Put $T=K+XL[X]$. 
	$T$ be a Noetherian if and only if $K\subset L$ be a normal extension.
\end{pr}

\begin{pr}
	\label{07}
	Let $K\subset L$ be fields extension such that $L^{G(L\mid K)}=K$. Put $T=K+XL[X]$. 	
	$T$ be a Noetherian if and only if $K\subset L$ be a normal extension.
\end{pr}

\begin{pr}
	\label{09}
	Let $T=K+XL[X]$ be Noetherian, where $K\subset L$ be fields. Assume
	$|G(L\mid K)|=[L\colon K]$ and any $K$-isomorphism $\varphi\colon M\to M$, where $\varphi(L)=L$ holds for every field $M$ such that $L\subset M$.
	$T$ be a Noetherian if and only if $K\subset L$ be a Galois extension. 
\end{pr}

\begin{pr}
	\label{10}
	Let $T=K+XL[X]$, where $K\subset L$ be fields such that $K=L^{G(L\mid K)}$. $T$ be a Noetherian if and only if $K\subset L$ be a Galois extension. 
\end{pr}

\begin{pr}
	\label{13}
	Let $K\subset L\subset M$ be fields such that $K$ be a perfect field. If $K+XL[X]$ and $L+XM[X]$ be Noetherian then $K\subset M$ be separable fields extension.
	
	\medskip
	
	Moreover, if we assume that any $K$-isomorphism $\varphi\colon M'\to M'$, where $\varphi(M)=M$ holds for every field $M'$ such that $M\subset M'$, then $K+XM[X]$ be a Noetherian. 
\end{pr}

\begin{pr}
	\label{14}
	Let $K\subset L\subset M$ be fields such that $M^{G(M\mid K)}=K$. If $K+XM[X]$ be Noetherian then $L\subset M$ be a normal fields extension. Moreover, $L+XM[X]$ be Noetherian.
\end{pr}

\begin{pr}
	\label{15}
	Let $K\subset L$ be extension fields such that $[L\colon K]=2$. Then $K+XL[X]$ be Noetherian. Moreover, if $L^{G(L\mid K)}=K$, then $K\subset L$ be a normal.
\end{pr}

\begin{tw}[\cite{Magid}, Theorem 1.2.]
	\label{tm1}
	Let $M$ be an algebraically closed field algebraic over $K$, and let $L$ such that $K\subseteq L\subseteq M$ be an intermediate field. Then the following are equivalent:
	\begin{itemize}
		\item[(a) ] $L$ is separable over $K$.
		\item[(b) ] $M\otimes_K L$ has no nonzero nilpotent elements.
		\item[(c) ] Every element of $M\otimes_K L$ is a unit times an idempotent.
		\item[(d) ] As an $M$-algebra $M\otimes_KL$ is generated by idempotents.
	\end{itemize}
\end{tw}

\begin{tw}[\cite{Magid}, Theorem 1.3.]
	\label{tm2}
	Let $M$ be an algebraically closed field containing $K$, and let $L$ be a field algebraic over $K$. Then the following are equivalent:
	\begin{itemize}
		\item[(a) ] $L$ is separable over $K$.
		\item[(b) ] $M\otimes_K L$ has no nonzero nilpotent elements.
		\item[(c) ] Every element of $M\otimes_K L$ is a unit times an idempotent.
		\item[(d) ] As an $M$-algebra $M\otimes_KL$ is generated by idempotents.
	\end{itemize}
\end{tw}	

Below we have conclusions from the above results.

\begin{tw}
	\label{tm3}
	In Theorems \ref{tm1} and \ref{tm2} if assume $L^{G(L\mid K)}=K$, then conditions (a) -- (d) are equivalent to
	\begin{itemize}
		\item[(e) ] $K+XL[X]$ be a Noetherian.
		\item[(f) ] $[L\colon K]<\infty$
		\item[(g) ] $K\subset L$ be an algebraic extension.
		\item[(h) ] $K\subset L$ be a Galois extension.
	\end{itemize}
\end{tw}

\begin{proof}
	(h)$\Rightarrow$(a) -- Obvious.
	
	\medskip
	
	(a)$\Rightarrow$(g)$\Rightarrow$(e)$\Rightarrow$(h) If $K\subset L$ be a separable extension, then be an algebraic extension. By Proposition \ref{02} $K+XL[X]$ be a Noetherian. By Proposition \ref{10} $K\subset L$ be a Galois extension.
	
	\medskip
	
	(e)$\Rightarrow$(f) -- Theorem \ref{01}.
\end{proof}

\begin{tw}
	\label{tm4}
	In Theorem \ref{tm3} if assume $K$ be a perfect field and $L^{G(L\mid K)}=K$, then conditions (a) -- (h) are equivalent to \\ 
	(g) $K\subset L$ be a normal extension.
\end{tw}

\begin{proof}
	(g)$\Rightarrow$(a) If $K\subset L$ be a normal extension, then be an algebraic extension. By definition perfect field $K\subset L$ be a separable extension.
	
	\medskip
	
	(h)$\Rightarrow$(g) Obvious.
\end{proof}

Proposition \ref{10}, Theorems \ref{tm3} and \ref{tm4} can be used to solve the inverse Galois problem. The inverse Galois problem concerns whether or not every finite group appears as the Galois group of some Galois extension of the rational numbers $\mathbb{Q}$. This problem, first posed in the early 19th century, is unsolved. 

\medskip

There is a lot of work. And it is enough to solve the problem for nonabelian groups. Thus, the following question arises:

\medskip

\noindent
{\bf Question:}\\
Can all the statements of this sections operate in noncommutative structures?

\medskip

And another question also arises regarding polynomial composites:

\medskip

\noindent
{\bf Question:}\\
Under certain assumptions for any type of $K\subset L$, we get that $K+XL[X]$ be a Noetherian ring. When can $K+XL[X]$ be isomorphic to any Noetherian ring?

\section{Generalized RSA cipher}
\label{R3}

In \cite{kk1} we have an information about how can we make a finite alphabet to an infinite alphabet?

\medskip

We can assign an appropriate number to each letter of the alphabet: $A = 0, B = 1, C = 2, D = 3, E = 4, F = 5, G = 6, H = 7, I = 8, J = 9, K = 10, L = 11, M = 12, N = 13, O = 14, P = 15, Q = 16, R = 17, S = 18, T = 19, U = 20, V = 21, W = 22 , X = 23, Y = 24, Z = 25$. So the alphabet is a finite set. The opposite side can easily decipher using the length of the alphabet. What if we extend this alphabet to an infinite set? In this situation, we can stay with the alphabet, but extend the length to infinity. So we have $A = 0 + 26k_0, B = 1 + 26k_1, C = 2 + 26k_2, \dots, Y = 24 + 26k_ {24}, Z = 25 + 26k_ {25} $, where $ k_0, k_1, \dots , k_ {25} \in\mathbb{N}_0 $. So, for example, the text $A B A C A B$ can be converted to $0\quad 1\quad 0\quad 2\quad 0\quad 1$, but also to $0\quad 1\quad 26\quad 54\quad 26\quad 53$. And we can give this number sequence to encrypt.

\medskip

{\bf Generating keys}

\medskip

Let's choose distinct prime ideals $P=(p)$ and $Q=(q)$ ($p$, $q$ are distinct primes) such that $N=PQ$ such that $|N|<|(x)|$, where $x$ is the length of the alphabet.

\medskip

Compute $\Phi (N)=(\varphi(n)):=(P-1)(Q-1)=(p-1)(q-1)$.

\medskip

Let's choose the ideal $E=(e)$ such that $e$ and $\varphi(n)$ are relatively primes ($\gcd (e, \varphi(n)=1)$) and $|\Phi(N)|< |E|\subsetneq (1)=\mathbb{N}_0$.

\medskip

We find the ideal $D=(d)$ such that $ED\equiv 1(\textrm{mod}\ \Phi(N))$.

\medskip

The public key is defined as the pair of ideals $(N, E)$, while the private key is the pair $(N, D)$.

\bigskip

{\bf Encryption and decryption}

\medskip

We encrypt the message $M=M_0M_1\dots M_r$ by calculation 
$$C_i\equiv M_iE (\textrm{mod}\ \Phi(N))$$.

\medskip

The encrypted message $C=C_0C_1\dots C_r$ is decrypted by formula 
$$M_1\equiv C_iD (\textrm{mod}\ \Phi(N)).$$

\bigskip
\section{Generalized Diffie–Hellman key exchange}
\label{RR5}

From \cite{kk1} recall a generalized Diffie-Hellman key exchange.

\medskip

First person F and second person S agree on the prime ideals $(p)$ and $(g)$ in $\mathbb{N}_0$ such that $|(p)|<|(g)|$.

\medskip

Person F chooses any secret $(a)$ in $\mathbb{N}_0$ and sends to person S $$(A)\equiv (g)(a) (\textrm{mod}\ (p)).$$

\medskip

Person S chooses any secret $(b)$ in $\mathbb{N}_0$ and sends to person F $$(B)\equiv (g)(b) (\textrm{mod}\ (p)).$$

\medskip

Person F compute $(s)\equiv (B)(a)(\textrm{mod}\ (p))$.

\medskip

Person S compute $(s)\equiv (A)(b)(\textrm{mod}\ (p))$.

\medskip

Person F and person S share a secret ideal $(s)$. This is because 
$$(s)\equiv (g)(a)(b)\equiv (g)(b)(a)(\textrm{mod}\ (p)).$$

\section{A key that is a fractional ideal}
\label{RR1}

In section \ref{RR1} and \ref{RR2} we have cryptosystems that use the Dedekind structure (\cite{kk2} in cooperation with M. Jankowska). 
My goal was not to create an entire cryptosystem based on the Dedekind structure. The first cryptosystem has a Dedekind structure in the key. The second cryptosystem has a Dedekind structure in two different alphabets. It is essential. This increases the security of our data. First of all, we use the fractional ideal structure. The definition itself is very interesting and motivated to apply.

\medskip

Let $A=\{a_0, a_1, \dots, a_n\}$ be an alphabet such that $|A|$ be a prime number.
Let $x\in\{2, 3, \dots, |A|\}$ be the value of one of the letters of the alphabet, $k\geqslant 2$ be an key. Then

$$y=xk \pmod{|A|},$$

where $y$ be the value of one of the letters of the alphabet be an encrypted letter.

\medskip

Now, assume we have encrypted letter $y$. Then we get a decrypted letter $x$ by a formula

$$x=(y+(k-d)\cdot |A|)\cdot k^{-1},$$

where $d$ be the remainder of dividing $y$ by $k$.

\begin{proof}
	\begin{align*}
	x&=\dfrac{y+(k-y\pmod k)|A|}{k}=\\
	&=\dfrac{xk\pmod{|A|}+((k-(xk\pmod{|A|}))\pmod k)|A|}{k}=x
	\end{align*}
\end{proof}

As proposed in \cite{kk1} (Introduction of section 3), this cipher can be generalized to a complete algebraic structure.
It is enough to adopt the infinite alphabet as in \cite{kk1}, $x$ be transformed into the principal ideal $(x)$, $k$ be transformed into the principal ideal $(k)$, $y$ into the principal ideal $(y)$. This way we get algebraic encryption where the key $(k)$ be the fractional ideal in the Dedekind's ring, in this case $\mathbb{Z}$.

\bigskip

\section{The alphabet as a fractional ideal}
\label{RR2}

Let $A$ be a set of characters. Assume $|A|$ is equal to any prime number. 

\medskip

Secretly establish a second alphabet $A'$ such that $A'\subset A$ with a prime length. 

\medskip

Let $m_1m_2m_3\dots m_n$ be a message, we want to encrypt.

\medskip

A secret short alphabet $A'$ divides a large public alphabet into zones.
We skip the extra characters such that $0$, $1$. So we have a clean alphabet from $2$. Let's move one over, so we have $1$. Suppose $p=|A|$, $q=|A'|$. We have $\big\lceil\dfrac{p}{q} \big\rceil$ zones.
Zero zone, includes the alphabet from $1$ to $q$. 
The first zone, i.e. the alphabet from $q+1$ to $2q$ and so on.
The last zone ($\big\lceil\dfrac{p}{q} \big\rceil-1$) includes the alphabet from $\big\lceil\dfrac{p}{q} \big\rceil q$ to $p$.

\medskip

Let's extend the message values with random numbers informing us about a given zone of a given letter (this information denote by $z_i$):
$$z_1m_1z_2m_2\dots z_nm_n$$

\medskip

Denote by $k$ the key. Multiply each value of the message (not the information about the zone) by $k$ and use the modulo $q$.

\medskip

Hence ciphertext is:
$$z_1d_1z_2d_2\dots z_nd_n,$$
where $d_1d_2\dots d_n$ be a encrypted message.

\medskip

Now let's decode the message.
$$z_1d_1z_2d_3\dots z_nd_n$$
by dividing it into blocks (each block contains a zone and a message).

\medskip

Let's apply the formula:
$$m_i=\dfrac{d_i+(z_i+t_i\cdot k)|A|}{k},$$
where $m_i$ is the decoded letter, $d_i$ encrypted letter, $z$ is a number satisfies a congruence $|A|^{-1}z_i\equiv d_i\pmod k$, $k$ be the key, $t$ be a zone.

\medskip

Of course, this cryptosystem can also be easily generalized by turning individual elements into ideals.

\bigskip
\section{Applications of polynomial composites in cryptology}
\label{R4}

Finally, we will show cryptosystems based on polynomial composites and monoid domains.

\begin{lm}
	Let $f=a_0+a_1X+\dots + a_{n-1}X^{n-1}+ \sum\limits_{j=n}^{m}a_jX^j$, $g=b_0+b_1X+\dots + b_{n-1}X^{n-1}+ \sum\limits_{j=n}^{m}b_jX^j$, where $a_i, b_i\in A_i$ for $i=0, 1, \dots, n-1$ and $a_j, b_j\in B$ for $j=n, n+1, \dots, m$. Then 
	$$fg\in A_0+XB[X].$$  
\end{lm}

Put $A_i, B_j$ $(i, j=0, 1, \dots, n-1)$ be different encryption systems. Then we have $f$ and $g$ are composition of encryption systems. No consider $B$. To improve security, let's fix that $\deg f=n-1$, $\deg g=n-k$, where $k\in\{2, \dots, n-1\}$. And such $f$, $g$ Alice and Bob agree before the message is sent. 

\medskip

Alice and Bob multiply these composites to form one. 
We have \\
$fg=(A_0+A_1X+\dots A_kX^k)(B_0+B_1X+\dots +B_lX^l)=A_0B_0+(A_0B_1+A_1B_0)X+\dots +A_kB_lX^{k+l}.$

Note that the sum and product of the encryption systems must be defined in the formula above. Definitions we leave Alice and Bob. But in this section we can put $S_iS_j: x\to (x)_{S_i}(x)_{S_j}$ and $S_i+S_j: x\to ((x)_{S_i})_{S_j}$. 
We can define the product and the sum of cryptosystems completely differently.

\medskip

So in the product we encrypt the letter as two letters, the first in the first system and the second in the second system. And in the sum we encrypt the letter using the first system and then the second system. Of course, we can define completely different, at our discretion.

\medskip

Assume that degree of $fg$ is $m$ and text to encrypt consists of more letters then $m+1$. 
Then we divide the text into blocks of length $m + 1$. 
We can assume that $fg(0)$ encrypts the first letter of each block. Expression at $X$ of $fg$ encrypts the second letter of each block, and expression at $X^2$ of $fg$ encrypts the third letter and so on.

\medskip

Now, let's see how to decrypt in this idea.

\medskip

Assume that we have an encrypted message $M_0M_1\dots M_n$. If our key is degree $m$, then we divide message on $m+1$ partition. And every partion divide to two. Every two letters are one letter of message. 

Earlier we define $S_iS_j: x\to (x)_{S_i}(x)_{S_j}$ and $S_i+S_j: x\to ((x)_{S_i})_{S_j}$. Then decryption of two letters $M_lM_{l+1}$ $(l=0, 2, 4, \dots )$ are $M_lM_{l+1}=(M_l)_{S_i}(M_{l+1})_{S_j}=N_{l,l+1}$ (one letter) and $M_l=((M_l)_{S_i})_{S_j}=(N_l)_{ij}$ (one letter). 

\medskip

The use of many cryptosystems in various configurations in a polynomial composite increases our security. The security here lies in the fact that the encrypted message is resistant to breaking under many cryptanalyst criteria.

\medskip

It is very easy to decrypt the message when you know the key.

\bigskip
\section{Applications of monoid domains in cryptology}
\label{R5}

Any alphabet of characters creates a finite set. Most ciphers are based on finite sets. But we can have the idea of using the infinite alphabet $\mathbb{A}$, although in reality they can be cyclical sets with an index that would mean a given cycle. For example, A$_0$ - $0$, B$_0$ - $1$, $\dots$ , Z$_0$ - $25$, A$_1$ - $0$, B$_1$ - $1$, $\dots$ , where A$_i$=A, $\dots$, Z$_i$=Z for $i=0, 1, \dots $. We see that this is isomorphic to a monoid $\mathbb{N}_0$ non-negative integers by a formula 

$$f\colon\mathbb{A}\to\mathbb{N}, f(m_i)=i.$$ 

\medskip

Then we can use a monoid domain by a map 
$$\varphi\colon\mathbb{A}\to F[\mathbb{A}], \varphi(m_0, m_1, \dots , m_n)=a_0X^{m_0}+\dots a_nX^{m_n}.$$

\medskip

We want to encrypt the message $m_0m_1m_2\dots m_n$ (the letters transform to numbers by a function $\varphi$).
We establish the secret key $X$.
Let $F$ be a field. We determine any coefficients from this field: $a_0$, $a_1$, $\dots$, $a_n$.
Then the message $m_0m_1m_2\dots m_n$ be transformed into a polynomial of the form:
$$a_0X^{m_0}+a_1X^{m_1}+\dots +a_nX^{m_n}.$$ 
We compute for $i=0, 1, \dots, n$:
$d_i=a_iX^{m_i}\pmod{|\mathbb{A}|}$ ($|\mathbb{A}|$ must be prime)
and then we have a decrypt message $d_0d_1\dots d_n$.

\medskip

To decrypt it we need to use a formula (for $i=0, 1, \dots, n$):
$$m_i=\log_X\dfrac{d_i}{a_i}\pmod{|\mathbb{A}|}.$$

\begin{proof}
	$$\log_X\dfrac{a_iX^{m_i}}{a_i}=m_i\pmod{|\mathbb{A}|}.$$
\end{proof}


\begin{thebibliography}{0}

\bibitem{0}
D.D. Anderson, D.F. Anderson, M. Zafrullah, Factorization in integral domains, {\it Journal of Pure and Applied Algebra}, {\bf 69}, (1990) 1--19.

\bibitem{1}
D.D. Anderson, D.F. Anderson, and M. Zafrullah, Rings between D[X] and K[X], {\it Houston J. of Mathematics}, {\bf 17}, (1991) 109--129.

\bibitem{y5}
D.F. Anderson and A. Ryckaert, The class group of $D + M$, {\it J. Pure Appl. Algebra}, {\bf 52}, (1988) 199--212.
	
\bibitem{y1}
J. Brewer and E. Rutter, D + M construction with general overrings, {\it Mich. Math. J.}, {\bf 23}, (1976) 33--42. 

\bibitem{y2}
D. Costa, J. Mott, and M. Zafrullah, The construction $D + XD_S[X]$, {\it J. Algebra}, {\bf 153}, (1978) 423--439.

\bibitem{zzz}
P.M. Eakin Jr., {\it The converse to a well known theorem on Noetherian rings}, Math. Ann. {\bf 177}, (1968) 278 -- 282.

\bibitem{y4}
M. Fontana and S. Kabbaj, On the Krull and valuative dimension of $D + XD_S[X]$ domains, {\it J. Pure Appl. Algebra}, {\bf 63}, (1990) 231--245.

\bibitem{2}
H. Kim, Factorization in monoid domains, {\it Comm. Algebra}, {\bf 29}, (2001) 1853--1869.

\bibitem{Magid}
A.R. Magid, {\it The Separable Galois Theory of Commutative Rings}, Chapman \& Hall/CRC, Boca Raton, 2014.

\bibitem{mm1}
\L{}. Matysiak, {\it On properties of composites and monoid domains}, Accepted for printing in Advances and Applications in Mathematical Sciences, \url{http://lukmat.ukw.edu.pl/files/On-properties-of-composites-and-monoid-domains--template-IJPAM-Italy-.pdf}, (2020).

\bibitem{mm2}
\L{}. Matysiak, {\it ACCP and atomic properties of composites and monoid domains}, Accepted for printing in Indian Journal of Mathematics, \url{http://lukmat.ukw.edu.pl/files/ACCP-and-atomic-properties-of-composites-and-monoid-domains.pdf}, (2020).

\bibitem{kk1}
Ł. Matysiak, {\it Generalized RSA cipher and Diffie-Hellman protocol}, J. Appl. Math.\& Informatics Vol.39 (2021), No. 1 - 2, pp. 93 -- 103

\bibitem{mm3}
\L{}. Matysiak, {\it On some properties of polynomial composites}, arXiv: 2104.09657, (2021).

\bibitem{mm4}
\L{}. Matysiak, {\it Polynomial composites and certain types of fields extensions}, arXiv: 2011.09904, (2021).

\bibitem{kk2}
M. Jankowska and \L{}. Matysiak, {\it A structure of Dedekind in the cryptosystem}, \url{http://lukmat.ukw.edu.pl/files/A-structure-of-Dedekind-in-the-cryptosystem.pdf}, (2021).

\bibitem{9}
T. Shah and W.A. Khan, On Factorization properties of monoid S and monoid domain D[S], {\it International Mathematical Forum}, {\bf 5}(18), (2010) 891--902.

\bibitem{y3}
M. Zafrullah, The $D + XD_S[X]$ construction from GCD-domains, {\it J. Pure Appl. Algebra}, {\bf 50}, (1988) 93--107.	


\end{thebibliography}
\end{document}